\documentclass[12pt,british]{article}
 \usepackage[active]{srcltx}
\usepackage{amsmath,amsfonts,amssymb,amscd}
\usepackage{latexsym,epsfig,graphicx}
\usepackage{babel}
\usepackage{theorem}
\usepackage[latin1]{inputenc}
\usepackage[T1]{fontenc}
\usepackage{enumerate}
\usepackage{amsbsy}
\usepackage[matrix,arrow]{xy}
\usepackage{endnotes}

\renewcommand{\epsilon}{\ensuremath{\varepsilon}}
\renewcommand{\phi}{\ensuremath{\varphi}}
\renewcommand{\to}{\ensuremath{\longrightarrow}}

\newcommand{\Z}{\ensuremath{\mathbb Z}}
\newcommand{\D}{\ensuremath{\mathbb D}}
\newcommand{\T}{\ensuremath{\mathbb{T}^2}}

\renewcommand{\ker}[1]{\ensuremath{\operatorname{\text{Ker}}({#1})}}

\makeatletter

\renewcommand{\p@enumii}{}

\def\@enum@{\list{\csname label\@enumctr\endcsname}%
           {\usecounter{\@enumctr}\def\makelabel##1{
\normalfont\ignorespaces\emph{{##1}~}}
\setlength{\labelsep}{3pt}
\setlength{\parsep}{0pt}
\setlength{\itemsep}{5pt}
\setlength{\leftmargin}{0pt}
\setlength{\labelwidth}{0pt}
\setlength{\listparindent}{\parindent}
\setlength{\itemindent}{0pt}
\setlength{\topsep}{3pt plus 1pt minus 1 pt}}}

\def\@map#1#2[#3]{\mbox{$#1 \colon\thinspace #2 \to #3$}}
\def\map#1#2{\@ifnextchar [{\@map{#1}{#2}}{\@map{#1}{#2}[#2]}}

\DeclareRobustCommand*\textsubscript[1]{\@textsubscript{\selectfont#1}}
\def\@textsubscript#1{{\m@th\ensuremath{_{\mbox{\fontsize\sf@size\z@#1}}}}}
\makeatother

\newcommand{\brak}[1]{\ensuremath{\left\{ #1 \right\}}}

\theoremheaderfont{\normalfont\scshape}

\newtheorem{thm}{Theorem}

\newtheorem{prop}[thm]{Proposition}
\newtheorem{cor}[thm]{Corollary}
\newtheorem{deft}{Definition}

\theorembodyfont{\rmfamily}

\newcommand{\reth}[1]{Theorem~\protect\ref{th:#1}}

\newcommand{\repr}[1]{Proposition~\protect\ref{prop:#1}}

\newcommand{\eop}{%
  \relax
  \ifvmode
    \noindent
  \else
    \unskip
    \hskip0pt plus-1fill\relax 
  \fi
  \vrule width0pt
  \nobreak
  \hfill 
  {\hspace*{\fill}$\Box$}
}

\newenvironment{proof}{\par\vspace{\partopsep}\noindent\emph{Proof.}}
{\eop\par\vspace{\parsep}}
\newenvironment{prooftext}[1]{\par\vspace{\parsep}\noindent{\emph{Proof of #1.}}}
{\eop\par\vspace{\parsep}}

\newtheorem{rem}{Remark}
\newtheorem{rems}[rem]{Remarks}

\newcommand{\sn}[1][n]{\ensuremath{\mathfrak{S}_{{#1}}}}

\newcommand{\MCGB}{\Gamma_{g,b}}

\newcommand{\PMCGB}{\Gamma_{g,b}^{n}}

\newcommand{\PPMCGB}{\mathrm{P}\Gamma_{g,b}^{n}}

\newcommand{\FPMCGB}{\mathrm{F}\Gamma_{g,b}^{n}}

\newcommand{\PFPMCGB}{\mathrm{P}\mathrm{F}\Gamma_{g,b}^{n}}

\newcommand{\diff}[1]{\mathrm{Diff}^{+}(#1)}

\newcommand{\build}[3]{\mathrel{\mathop{\kern 0pt#1}\limits_{#2}^{#3}}}

\newcounter{liste}

\makeatletter
\renewcommand{\@makeenmark}{\hbox{\,\@theenmark}}
\renewcommand{\enoteformat}{\parindent=1.5em\leavevmode\llap{\hbox{\bf \@theenmark}}}
\makeatother

\allowdisplaybreaks

\begin{document}

 \title{Surface framed braids}
 \author{Paolo Bellingeri and Sylvain Gervais}


\date{\empty}
\maketitle

 \begin{abstract}
In this paper we introduce the framed pure braid group on $n$ strands of an oriented surface, a  
topological generalisation of the pure braid group
$P_n$. We give different equivalents definitions for framed pure braid groups and we study exact sequences
relating these groups  with other generalisations of $P_n$, usually called  surface pure braid groups.
The notion of surface framed braid groups is also introduced.
 \end{abstract}

\begingroup
 \renewcommand{\thefootnote}{}
 \endgroup


\section{Introduction}
Let $\Sigma=\Sigma_{g,b}$ be an oriented surface of genus $g$ with $b$ boundary components (we will  note 
$\Sigma_{g}:=\Sigma_{g,0}$) and let $\mathcal{P}=\{p_1, \dots,p_n\}$ be a set of $n$ distinct points 
(\emph{punctures}) in the interior of $\Sigma$.
Let $C_n(\Sigma)=\Sigma^n \setminus \Delta$, where $\Delta$ is  the  set of 
$n$-tuples $\underline{x}=(x_1, \dots, x_n)$ for which $x_i=x_j$ for some $i \not= j$. The fundamental group  
$\pi_1(C_n(\Sigma),\underline{p})$ is called  \emph{pure braid group on $n$ strands of the 
surface} $\Sigma$; it shall be denoted by $P_n(\Sigma)$.\\
The symmetric group $\sn$ acts freely on $C_{n}(\Sigma)$ by permutation of 
coordinates. We denote $\widehat{C_{n}}(\Sigma)$ the quotient space $C_{n}(\Sigma)/\sn$. The fundamental 
group of $\widehat{C_{n}}(\Sigma)$ is called  \emph{braid group on $n$ strands of the 
surface} $\Sigma$; it shall be denoted by $B_n(\Sigma)$.\\
Since the projection map $C_{n}(\Sigma)\to\widehat{C_{n}}(\Sigma)$ is a regular covering space with transformation 
group $\sn$, one has the following exact sequence:

\begin{equation*}\label{eq:covering sequence}
1\to P_{n}(\Sigma)\to B_{n}(\Sigma)\to \sn\to 1.
\end{equation*}

On the other hand, from the homotopy exact sequence associated to the fibration $C_{n+m}(\Sigma)\to C_{n}(\Sigma)$, we 
get~(\cite{FN,Bir}) an exact sequence: 

\begin{equation*}\label{eq:pure braid sequence}
(SPB) \quad 1\to P_{m}\left(\Sigma \setminus n \; \mbox{points} \;\right)\to P_{n+m}(\Sigma)\to 
P_{n}(\Sigma)\to 1
\end{equation*}
when $\Sigma$ has positive genus or genus equal to zero and non-empty boundary (in the case of the sphere 
such a sequence holds for $n+ m\ge 4$). In the following we will denote this sequence by $(SPB)$ 
(\emph{Surface Pure Braids}).

When  $\Sigma$ has boundary, the (SPB) sequence has an evident geometric section, which corresponds to add 
$m$ punctures "at the infinity" (see~\cite{B,GG}); if $\Sigma$ is closed and different from the sphere and the torus, the
(SPB) sequence  splits if and only if  $n=1$ (see~\cite{GG} for the general case and~\cite{BB} for an algebraical section
in the case  $n=1$).
In the case of the disk this sequence has some additional 
features~\cite{FR} and it is  a powerful tool in the study of finite type invariants for links~\cite{pap}.
In the general case the (SPB) sequence can be used to find a group presentation for  $P_n(\Sigma)$
(see for instance \cite{B}).

\smallskip
In this paper we introduce framed braid groups $FB_{n}(\Sigma)$ and $FP_{n}(\Sigma)$ of a surface $\Sigma$, 
which generalise respectively framed braid groups introduced in~\cite{KS} and framed pure braid groups 
considered in Theorem 5.1 of~\cite{MM}. These groups turn out also to be related to 
generalisations of Hilden groups introduced in~\cite{BC}. For surfaces of genus greater than $1$, with boundary or closed,
we give three equivalent definitions of these groups: in terms of configuration 
spaces, as subgroups of mapping class 
groups (Section~\ref{section:definitions})  and  as subgroups  of 
braid groups of surfaces (Section~\ref{section:frameddehn} and~\ref{section:framesbraidsas2nstrands}). We  prove that, 
when the surface is closed and of genus greater than $1$, these 
groups are non trivial central extensions of 
surface pure braid groups (Section~\ref{section:classictoframed}) and 
we provide a group presentation for $FP_{n}(\Sigma)$ (Section~\ref{section:presframed}, Theorem~\ref{presentationframed})
and therefore for $FB_n(\Sigma)$ (Theorem~\ref{presentationbraidframed}). We show also that the 
sequence~($SPB$) extends naturally to a sequence on framed braids, that we will call \emph{framed surface 
pure braid sequence} (denoted by~(FSPB)) and which splits even in the 
case of closed surfaces (Section~\ref{section:framed}, Theorem~\ref{th:framedsequence}).
In the case of the torus the proposed definitions are not equivalent and let arise different notions of framings:
this case will treated separately in the last Section.


\section{Framed braids: possible definitions}\label{section:definitions}

\subsection{Framed braids via configuration spaces}

Let $U\Sigma$ be the unit tangent bundle of $\Sigma$ and $\pi:U\Sigma\to\Sigma$ be the natural projection. We 
denote by $F_{n}(\Sigma)$ the subspace $(\pi^{n})^{-1}(C_{n}(\Sigma))$ of $(U\Sigma)^n$ and fix a unit tangent vector 
$v_{i}$ of $\Sigma$ at $p_{i}$ such that $F_{n}(\Sigma)$ is based at 
$\underline{v}=(p_{i},v_{i})_{i=1,\ldots,n}$. The symmetric group $\sn$ acts freely on $F_{n}(\Sigma)$: we 
denote $\widehat{F_{n}}(\Sigma)$ the quotient space $F_{n}(\Sigma)/\sn$.

\begin{deft}\label{defframed}
The \emph{pure framed braid group} $FP_{n}(\Sigma)$ on $n$ strands of $\Sigma$ is the fundamental group of 
$F_{n}(\Sigma)$. The \emph{framed braid group} on $n$ strands of $\Sigma$ is the fundamental group of 
$\widehat{F_{n}}(\Sigma)$.
\end{deft}

\pagebreak[4]\noindent
Thus, a framed braid can be seen as a family of $n$ continuous paths $b_{i}:[0,1]\to U\Sigma$ for $i=1,
\ldots, n$
such that :
\begin{list}{\arabic{liste})}{\usecounter{liste}\leftmargin=25mm}
\item $b_{i}(0)=(p_{i},v_{i})$ for all $i\in\brak{1,\ldots,n}$;
\item $\exists\sigma\in\sn$ such that $b_{i}(1)=(p_{\sigma(i)},v_{\sigma(i)})$ for all
$i\in\brak{1,\ldots,n}$;
\item $\pi b_{i}(t)\neq\pi b_{j}(t)$ when $i\neq j$ for any $t\in [0,1]$.
\end{list}

Since the projection map $F_{n}(\Sigma)\to\widehat{F_{n}}(\Sigma)$ is a regular covering space with 
transformation group $\sn$, the framed braid group and the pure framed braid group are related by the 
following exact sequence:

\begin{equation}\label{eq:framed covering sequence}
1\to FP_{n}(\Sigma)\to FB_{n}(\Sigma)\to \sn\to 1
\end{equation}

\subsection{Framed braids as  mapping classes}\label{section:mapping classes}

In this section we will give an interpretation of framed braid groups of an oriented  surface different from the sphere and the torus in terms 
of mapping classes.

\subsubsection{Notations} Let $\diff{\Sigma_{g,b}}$ denote the group of orientation preserving diffeomorphisms of $\Sigma_{g,b}$ which are the 
identity on the boundary. Recall that the  \emph{mapping class group} of $\Sigma_{g,b}$, denoted $\MCGB$, is defined to be 
$\pi_{0}(\diff{\Sigma_{g,b}})$, where $\diff{\Sigma_{g,b}}$ is equipped with the compact open topology. Note that we will 
denote the composition in the mapping class groups from left to right\footnote{We do this in order to have 
the same group-composition in braid groups and mapping class groups.}. In the following, greek letters will be 
used to denote simple closed curves on $\Sigma_{g,b}$ and if $\alpha$ is such a curve, $\tau_{\alpha}$ will denote the Dehn 
twist along $\alpha$.\\
We shall also consider different subgroups of $\diff{\Sigma_{g,b}}$ and associated mapping class groups:

\medskip
$\bullet$ $\diff{\Sigma_{g,b},\mathcal{P}}=\{h\in\diff{\Sigma_{g,b}}\,/\,\exists\sigma\in\sn,\ h(p_{i})=p_{\sigma(i)}\}$,
and the \emph{punctured mapping class group} $\PMCGB=\pi_{0}(\diff{\Sigma_{g,b},\mathcal{P}})$;

\medskip
$\bullet$ $\diff{\Sigma_{g,b},\underline{p}}=\{h\in\diff{\Sigma_{g,b},\mathcal{P}}\,/\,h(p_{i})=p_{i}\}$ 
and the \emph{pure punctured mapping class group} $\PPMCGB=\pi_{0}(\diff{\Sigma_{g,b},\underline{p}})$;

\medskip
$\bullet$ $\diff{\Sigma_{g,b},\mathcal{V}}=\{h\in\diff{\Sigma_{g,b}}\,/\,\exists\sigma\in\sn,\ h(p_{i})=p_{\sigma(i)}
\ \mathrm{and}\ d_{p_{i}}h(v_{i})=v_{\sigma(i)}\}$ (where 
$\mathcal{V}=\{(p_{1},v_{1}),\ldots,(p_{n},v_{n})\}$  is a set of $n$ distinct points on $\Sigma_{g,b}$
equipped with $n$ unit tangent vectors $v_1, \ldots, v_n$) and the \emph{framed punctured mapping class group}
$\FPMCGB=\pi_{0}(\diff{\Sigma_{g,b},\mathcal{V}})$;

\medskip
$\bullet$ $\diff{\Sigma_{g,b},\underline{v}}=\{h\in\diff{\Sigma_{g,b},\mathcal{V}}\,/\,h(p_{i})=p_{i}\ \mathrm{and}\ 
d_{p_{i}}h(v_{i})=v_{i}\}$,  and the \emph{pure framed punctured mapping class group} 
$\PFPMCGB=\pi_{0}(\diff{\Sigma_{g,b},\underline{v}})$.

\begin{rem}
A preserving orientation diffeomorphism $h$ of $\Sigma_{g,b}$ such that $h(p_{i})=p_{i}$ and 
$d_{p_{i}}h(v_{i})=v_{i}$ is isotopic to a  
diffeomorphism which is equal to identity on a small disc around 
$p_{i}$. Thus, the two groups $\PFPMCGB$ and $\Gamma_{g,b+n}$ are 
isomorphic.
\end{rem}

\medskip
Let us also recall that if $\Sigma'$ is a subsurface of $\Sigma$ and $i:\Sigma'\hookrightarrow\Sigma$ is 
the inclusion map, there is a canonical morphism $i_{*}$ from the mapping  class group $\Gamma(\Sigma')$ of 
$\Sigma'$ to the mapping  class group $\Gamma(\Sigma)$ of $\Sigma$, which consists in extending each diffeomorphism $h$ of $\Sigma'$ by identity on 
$\Sigma\setminus\Sigma'$. When $\Sigma$ is a genus $g$ surface with $b$ boundary components and 
$\Sigma\setminus\Sigma'$ is a collection of $n$ disjoint discs (see Figure~\ref{figure:embedding}), 
we shall denote by $\lambda_{g,b}^{n}:\Gamma_{g,b+n}\to\Gamma_{g,b}$ this morphism ($\lambda_{g}^{n}$ when 
$b=0$).

\begin{figure}[h]
\begin{center}
\psfig{figure=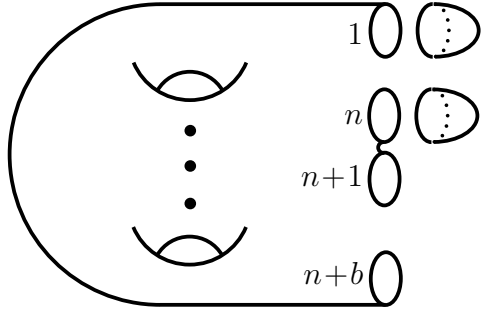}
\caption{\label{figure:embedding}
The embedding $\Sigma_{g,b+n}\hookrightarrow\Sigma_{g,b}$.}
\end{center}
\end{figure}

\subsubsection{Braids groups as mapping class groups}

Braid groups of surface are related to mapping class groups as follows (see \cite{Bir,S1}):

\begin{prop}\label{prop:braids as mapping classes}
Let $n\geq 1$. Let $\psi_{n}:\PMCGB\to\MCGB$ and $\varphi_{n}:\PPMCGB\to\MCGB$ be the homomorphisms induced by 
the forgetting map $\diff{\Sigma_{g,b},\mathcal{P}} \to 
\diff{\Sigma_{g,b}}$ and $\diff{\Sigma_{g,b},\underline{p}} \to \diff{\Sigma_{g,b}}$.\\
1) If $(g,b)\notin \{(0,0),(1,0)\}$, $\ker{\psi_{n}}$ and $\ker{\varphi_{n}}$ are 
respectively isomorphic to $B_{n}(\Sigma_{g,b})$ and $P_{n}(\Sigma_{g,b})$.

\noindent
2) When $b=0$ and $g\in \{0,1\}$, $\ker{\psi_{n}}$ and $\ker{\varphi_{n}}$ are respectively isomorphic to 
$B_{n}(\Sigma_{g})/Z(B_{n}(\Sigma_{g}))$ and $P_{n}(\Sigma_{g})/Z(P_{n}(\Sigma_{g}))$ (where $Z(G)$ is the center of the group 
$G$).
\end{prop}

\begin{rems}\hfill\break
1) Let $\T$ be a torus; one has 
$Z(B_{n}(\T))=Z(P_{n}(\mathbb{T}^{2}))=\Z^{2}$ 
(see~\cite{PR1}).

\medskip\noindent
2) Let  $S^{2}$ be a sphere; one has $B_{1}(S^{2})=P_{1}(S^{2})=1$, $P_{2}(S^{2})=1$, 
$B_{2}(S^{2})=\Z/2\Z$ and for $n\geq 3$, $Z(B_{n}(S^{2}))=Z(P_{n}(S^{2}))\approx \Z/2\Z$ (see~\cite{GV}).
\end{rems}

We  can provide  a similar result for framed braids:

\begin{prop}\label{prop:framed braids as mapping classes}
Let $n\geq 1$. Let $\Psi_{n}:\FPMCGB\to\MCGB$ and $\Phi_{n}:\PFPMCGB\to\MCGB$ be the homomorphism induced by 
the maps which forget the tangent vectors and the punctures.
If $(g,b)\notin \{(0,0),(1,0)\}$, $\ker{\Psi_{n}}$ and $\ker{\Phi_{n}}$ are 
respectively isomorphic to $FB_{n}(\Sigma_{g,b})$ and $FP_{n}(\Sigma_{g,b})$.
\end{prop}

\begin{proof}
Following~\cite{Bir}, we consider the evaluation map $Ev:\diff{\Sigma_{g,b}}\to F_{n}(\Sigma_{g,b})$ defined by 
$Ev(h)=\bigl(h(p_{i}),d_{p_{i}}h(v_{i})\bigr)$. It is a locally trivial fibering with fiber 
$\diff{\Sigma_{g,b},\underline{v}}$. The long exact sequence of homotopy groups of this fibration gives the 
following exact sequence: 
\begin{multline*}
\cdots\to\pi_{1}\bigl(\diff{\Sigma_{g,b}}\bigr)\stackrel{Ev_{\ast}}{\to}\pi_{1}\bigl(F_{n}(\Sigma_{g,b})\bigr)
\stackrel{\partial_{\ast}}{\to}\pi_{0}\bigl(\diff{\Sigma_{g,b},\underline{v}}\bigr)\\
\stackrel{\Phi_{n}}{\to}\pi_{0}\bigl(\diff{\Sigma_{g,b}}\bigr)
\stackrel{Ev_{\ast}}{\to}\pi_{0}\bigl(F_{n}(\Sigma_{g,b})\bigr).
\end{multline*}
Now, $\pi_{0}\bigl(F_{n}(\Sigma_{g,b})\bigr)$ is trivial, and if $\Sigma_{g,b}$ is not the sphere $S^{2}$ or the torus 
$\mathbb{T}^{2}$, $\diff{\Sigma_{g,b}}$ is contractible (see~\cite{S2}). Thus, we get the required result for 
$\Phi_{n}$.

\noindent
The proof for $\Psi_{n}$ is analogous  and is 
left to the reader; it suffices to consider the evaluation map $\widehat{E_v}:\diff{\Sigma_{g,b}}\to 
\widehat{F_{n}}(\Sigma_{g,b})$ defined by $\widehat{Ev}(h)=\overline{\bigl(h(p_{i}),d_{p_{i}}h(v_{i})\bigr)}$.

\end{proof}

The definition of framed braid groups in terms of mapping classes and tangent vectors provided in \repr{framed braids as mapping classes}
was  introduced in~\cite{OS} as the definition of framed braid groups of surfaces with one boundary component.

\medskip
We can provide an equivalent definition of pure framed braid groups as the kernel of the mapping induced 
by the inclusion of a surface with boundary components into another one.

\begin{cor}\label{cor:capping}
For all $(g,b)$ distinct from $(0,0)$ and $(1,0)$, the kernel of the morphism 
$\lambda_{g,b}^{n}:\Gamma_{g,b+n}\to\Gamma_{g,b}$ induced by capping $n$ boundary components is isomorphic to 
the pure framed braid group $ FP_{n}(\Sigma_{g,b})$.
\end{cor}
\begin{proof}
Via the isomorphism $\PFPMCGB\approx\Gamma_{g,b+n}$, the homorphism $\Phi_{n}$ coincides with 
$\lambda_{g,b}^{n}$.
\end{proof}


\section{The (FSPB)  sequence}\label{section:framed}
In this section, we prove the existence of a framed version of the $(SPB)$ sequence. We also prove that it splits even in the case of 
closed surfaces of genus greater or equal than $2$.

\begin{thm}\label{th:framedsequence}
For $g\geq 2$, $b\geq 0$, $n\geq 0$ and $m\geq 0$, one has the following splitting exact sequence :
\begin{equation*}
(FSPB) \qquad 1\to FP_{m}(\Sigma_{g,b+n})\to 
FP_{n+m}(\Sigma_{g,b})\stackrel{\alpha_{n,m}}{\to}FP_{n}(\Sigma_{g,b})\to 1
\end{equation*}
where $\alpha_{n,m}$ consists in forgetting the first $m$ strands.
\end{thm}

\begin{proof}
Using Proposition~\ref{prop:framed braids as mapping classes} and Corollary~\ref{cor:capping}, one has the following commutative 
diagram with exact rows and columns:
$$
\xymatrix{%
&      & 1 \ar[d]                                 & 1 \ar[d]                    & \\
1\ar[r]& FP_{m}(\Sigma_{g,b+n}) \ar[r] \ar@{=}[d] & FP_{n+m}(\Sigma_{g,b}) \ar[r]^{(\lambda^{m}_{g,b+n})_{\mid}} 
\ar[d] & FP_{n}(\Sigma_{g,b}) \ar[r] \ar[d] & 1\\
1 \ar[r] & FP_{m}(\Sigma_{g,b+n}) \ar[r] & \Gamma_{g,b+n+m} \ar[d]^{\lambda^{n+m}_{g,b}}\ar[r]^{\lambda^{m}_{g,b+n}} 
& \Gamma_{g,b+n} \ar[d]^{\lambda^{n}_{g,b}} \ar[r] & 1\\
&  & \Gamma_{g,b} \ar@{=}[r]\ar[d] & \Gamma_{g,b}  \ar[d]\\
& & 1 & 1 &
}$$
where $(\lambda^{m}_{g,b+n})_{\mid}$ is the restriction of 
$\lambda^{m}_{g,b+n}$ to $FP_{n+m}(\Sigma_{g,b})$. Since, via the 
isomorphism $FP_{n+m}(\Sigma_{g,b})\approx\ker{\lambda^{n+m}_{g,b}}$, 
forgetting $m$ strands corresponds to capping $m$ boundary 
components, the first row is the required exact sequence.\\
The splitting of the (FSBP) sequence is obtained considering an embedding  $\iota$ of $\Sigma_{g,b+n}$ into 
$\Sigma_{g,b+n+m}$ as in~figure~\ref{figure:Section}. The induced morphism  
$\iota_{*}:\Gamma_{g,b+n}\to\Gamma_{g,b+n+m}$ satisfies 
$\lambda^{m}_{g,b+n}\circ\iota_*=\mathrm{Id}_{\Gamma_{g,b+n}}$  
and its restriction $\iota_{* \mid}:FP_{n}(\Sigma_{g,b})\to  FP_{n+m}(\Sigma_{g,b})$
is a section for $(\lambda^{m}_{g,b+n})_{\mid}$.

\begin{figure}[h]
\begin{center}
\psfig{figure=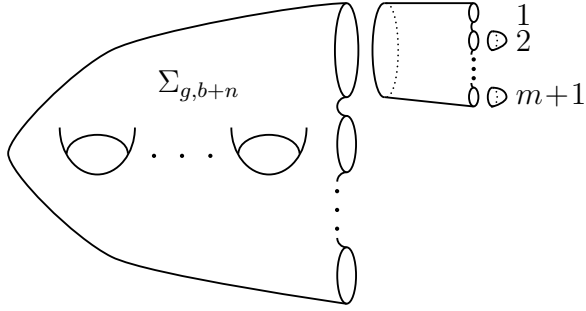}
\caption{\label{figure:Section} $\iota$ is a section of 
$\lambda_{b+n}^{m}$.}
\end{center}
\end{figure}

\end{proof} 

\begin{rem}
Note that the splitting of (FSPB) sequence consists in cabling the first framed strand while the splitting of (SPB) sequence in the case of surfaces with boundary 
consists in  adding $m$ strands  at the infinity. If $\Sigma$ has boundary, adding $m$ framed strands at the infinity gives another 
splitting of (FSPB) sequence. 
\end{rem}


\section{Framed braids vs classical braids}\label{section:classictoframed}

The following Theorem shows that the group $FP_{n}(\Sigma_{g})$ is a non trivial central extension 
of $P_{n}(\Sigma_{g})$ by $\Z^{n}$.

\begin{thm} \label{th:extension centrale par Zn cas pure}
Let $\Sigma_{g,b}$ a surface of genus $g\ge2$ with $b$ boundary components.

\noindent 1) If $\Sigma_{g,b}$ has boundary, the framed pure braid group 
$FP_{n}(\Sigma_{g,b})$ is isomorphic to $\Z^{n}\times P_{n}(\Sigma_{g,b})$.\\
2) If $\Sigma_{g}$ is a closed surface, there is a non-splitting central extension
\begin{equation}\label{eq:pure framed braid and pure braid}
 1\to \Z^{n}\to FP_{n}(\Sigma_{g})\stackrel{\beta_{n}}{\to}  P_{n}(\Sigma_{g})\to 1
\end{equation}
where $\beta_{n}$ is the morphism induced by the projection map $F_{n}(\Sigma_{g})\to C_{n}(\Sigma_{g})$ 
({\em i.e.} $\beta_{n}$ consists in forgetting the framing).\\
3) In the two cases, $F_{n}(\Sigma_{g,b})$ is an Eilenberg-Maclane space of type $(FP_{n}(\Sigma_{g,b}),1)$.
\end{thm}

\begin{proof}
If $\Sigma_{g,b}$ has boundary, $\Sigma_{g,b}$ is parallelizable and
the unit tangent bundle $U\Sigma_{g,b}$ is homeomorphic to $S^{1}\times\Sigma_{g,b}$. Thus, 
$F_{n}(\Sigma_{g,b})$ is homeomorphic to $(S^{1})^{n}\times C_{n}(\Sigma_{g,b})$ and $FP_{n}(\Sigma_{g,b})$ 
is isomorphic to $\Z^{n}\times P_{n}(\Sigma_{g,b})$. Furthermore, since $S^{1}$ and $C_{n}(\Sigma_{g,b})$ are 
Eilenberg-Maclane spaces (see~\cite{PR1}), $F_{n}(\Sigma_{g,b})$ is also one.

Now, suppose that $\Sigma_{g}$ is a closed surface of genus greater than $1$ and consider the exact sequence of 
homotopy groups of the locally trivial fibrations (with fiber
$(S^{1})^{n}$) $F_{n}(\Sigma_{g})\to C_{n}(\Sigma_{g})$ :
\setlength{\multlinegap}{0mm}
\begin{multline*}
\cdots\to\pi_{n}\bigl((S^{1})^{n}\bigr)\to\pi_{n}\bigl(F_{n}(\Sigma_{g,b})\bigr)\to\pi_{n}\bigl(
C_{n}(\Sigma_{g,b})\bigr)\to
\pi_{n-1}\bigl((S^{1})^{n}\bigr)\to\cdots\\ \cdots\to\pi_{2}\bigl( 
C_{n}(\Sigma_{g,b})\bigr)\to\pi_{1}\bigl((S^{1})^{n}\bigr)\to\pi_{1}\bigl(
F_{n}(\Sigma_{g,b})\bigr)\stackrel{\beta_{n}}{\to}\pi_{1}\bigl( C_{n}(\Sigma_{g,b})\bigr)\to 1
\end{multline*}
Since $C_{n}(\Sigma_{g})$ is an Eilenberg-Maclane spaces (see~\cite{PR1}), this sequence leads to the 
sequence (\ref{eq:pure framed braid and pure braid}) and to the third point of the Theorem.\\
Now, using Propositions~\ref{prop:braids as mapping classes} and~\ref{prop:framed braids as mapping 
classes} one has the following commutative diagram with exact rows and columns: 
\addtocounter{equation}{1}
$$
\xymatrix{%
&&1 \ar[d]  & 1 \ar[d]& \\
1 \ar[r] & \Z^{n} \ar[r]\ar@{=}[d] & FP_{n}(\Sigma_{g}) \ar[r] \ar[d] & P_{n}(\Sigma_{g}) \ar[r] \ar[d] &
1 \ \ \ \ \ \ \ \ \ 
(\ref{eq:pure framed braid and pure braid})\\
1 \ar[r] & \Z^{n} \ar[r] &\Gamma_{g,n} \ar[d]^{\Phi_{n}}  \ar[r] & \mathrm{P}\Gamma_{g}^{n} 
\ar[d]^{\varphi_{n}} \ar[r] & 1 \ \ \ \ \ \ \ \ \ (\arabic{equation})\\
&& \Gamma_{g} \ar@{=}[r]\ar[d] & \Gamma_{g}\ar[d]\\
&& 1  & 1
}$$
where the exact sequence~(\arabic{equation}) is obtained by capping 
$n$ boundary components of $\Sigma_{g,n}$ by $n$ punctured discs 
(see~\cite{PR2}). In this sequence, $\Z^{n}$ can be seen as a 
subgroup of $\Gamma_{g,n}$ generated by Dehn twists along curves 
parallel to boundary components, thus~(\arabic{equation}) is a 
central extension of $\mathrm{P}\Gamma_{g}^{n}$. 
Therefore,~(\ref{eq:pure framed braid and pure braid}) is a central 
extension of $P_{n}(\Sigma_{g})$.\\
To conclude, we shall prove that this sequence does not split. First we remark that for $n=1$ we obtain the 
canonical projection of the fundamental group of the unit tangent space into the fundamental group of the 
surface which  splits if and only if the surface is not closed (we are considering the case $g\ge2$). 
Thus, let $n\ge2$  and let us consider the following diagram:

$$
\xymatrix{%
1 \ar[r] & \Z^{n} \ar[r]\ar@{=}[d] & FP_{n+1}(\Sigma_g) \ar[r]^{\beta_{n+1}} \ar@<2pt>[d]^{(\lambda_{g,n}^{1})_{\mid}}  &
P_{n+1}(\Sigma_g) \ar[r] \ar[d]^{p_\ast} & 1 \\
1 \ar[r] & \Z^{n} \ar[r]           & FP_{n}(\Sigma_{g}) \ar[r]^{\beta_{n}}  \ar@<2pt>[u]^{(\iota_\ast)_{\mid}} &
P_{n}(\Sigma_{g}) \ar[r]                 & 1
}
$$
where $\iota_\ast$ is the section of $\lambda_{g,n}^{1}$ described in 
figure~\ref{figure:Section} and $p_\ast$ the map wich consists in forgetting one strand. One can easily verify that the 
diagram is commutative.  Suppose that there is a section $s_{n}: P_{n}(\Sigma_g) \to FP_{n}(\Sigma_{g})$ for 
$\beta_{n}$. 

The composition $\beta_{n+1} \circ \iota_\ast \circ
s_{n}: P_{n}(\Sigma_{g})\to P_{n+1}(\Sigma_g) $ is therefore a section for $p_\ast$: in fact $p_\ast \circ\beta_{n+1} \circ
\iota_\ast \circ s_{n}
= \beta_{n} \circ \lambda_{g}^{n}  \circ \iota_\ast \circ s_{n}= Id_{P_{n}(\Sigma_{g})}$. This is impossible,
since
$p_\ast$ has no section when $n\ge2$ and  $g>1$~\cite{GG}.
\end{proof}

\begin{rem}
The splitting of sequence~(\ref{eq:pure framed 
braid and pure braid}) in the case of surfaces with boundary was proven  in Lemma 19 of \cite{BGG} in a
combinatorial way. 
\end{rem}

\begin{prop} \label{prop:extension centrale par Zn}
Let $\Sigma_{g,b}$ be a surface different from the sphere $S^{2}$ and the torus $\T$.\\
1) there is an exact sequence
\begin{equation}\label{eq:framed braid and braid}
1\to \Z^{n}\to FB_{n}(\Sigma_{g,b})\stackrel{\widehat{\beta}_{n}}{\to}  B_{n}(\Sigma_{g,b})\to 1\; ,
\end{equation}
where $\widehat{\beta}_{n}$ consists in forgetting the framing. This sequence 
 splits if $\Sigma_{g,b}$ has non-empty boundary.\\
2) $\widehat{F_{n}}(\Sigma_{g,b})$ is an Eilenberg-Maclane space of type $(FB_{n}(\Sigma_{g,b}),1)$.
\end{prop}

\begin{proof} Using local trivialisations of the fibering $F_{n}(\Sigma_{g,b})\to C_{n}(\Sigma_{g,b})$ and
the covering 
spaces $C_{n}(\Sigma_{g,b})\to\widehat{C_{n}}(\Sigma_{g,b})$ and
$F_{n}(\Sigma_{g,b})\to\widehat{F_{n}}(\Sigma_{g,b})$, one can 
easily see that $\widehat{F_{n}}(\Sigma_{g,b})\to \widehat{C_{n}}(\Sigma_{g,b})$ is a locally trivial bundle
with 
fibre $(S^{1})^{n}$. The exact sequence of homotopy groups of this fibration is
\setlength{\multlinegap}{0mm}
\begin{multline*}
\cdots\to\pi_{n}\bigl((S^{1})^{n}\bigr)\to\pi_{n}\bigl(\widehat{F_{n}}(\Sigma_{g,b})\bigr)\to
\pi_{n}\bigl(\widehat{C_{n}}(\Sigma_{g,b})\bigr)\to\pi_{n-1}\bigl((S^{1})^{n}\bigr)\to\cdots\\
\cdots\to\pi_{2}\bigl( 
\widehat{C_{n}}(\Sigma_{g,b})\bigr)\to\pi_{1}\bigl((S^{1})^{n}\bigr)\to\pi_{1}\bigl(
\widehat{F_{n}}(\Sigma_{g,b})\bigr) \stackrel{\widehat{\beta}_{n}}{\to} \pi_{1}\bigl(\widehat{C_{n}}(\Sigma_{g,b})\bigr)\to 1
\end{multline*}
Since $\widehat{C_{n}}(\Sigma_{g,b})$ is an Eilenberg-Maclane space (see~\cite{PR1}), we get the
required 
exact sequence and the second point of the proposition. If $\Sigma_{g,b}$ is parallelizable, any section
$s:C_{n}(\Sigma_{g,b})\to 
F_{n}(\Sigma_{g,b})$ induces a section
$\hat{s}:\widehat{C_{n}}(\Sigma_{g,b})\to\widehat{F_{n}}(\Sigma_{g,b})$ which gives the 
splitting of~(\ref{eq:framed braid and braid}).
\end{proof}


\section{Presentations of pure framed braid groups}\label{section:presframed}

Now, let us look for a presentation of pure framed braid groups. By~\reth{extension centrale par Zn cas pure}, 
the group $FP_{n}(\Sigma_{g,b})$ is isomorphic to the direct product $\Z^{n}\times P_{n}(\Sigma_{g,b})$ if 
$b\geq 1$. Thus, with the following theorem (see~\cite{B}), we get a 
presentation of $FP_{n}(\Sigma_{g,b})$.

\begin{thm} \label{th:purepres}
Let $\Sigma_{g,b}$ be a compact, connected, orientable surface
of genus $g\ge 1$ with $b$ boundary components, $b\geq 1$. The group $P_n(\Sigma_{g,b})$ admits the following
presentation:
\begin{enumerate}
\item[\textbf{Generators:}] $\{A_{i,j}\; | \;1 \le i \le 2g+b+n-2, 2g 
+b\le j \le 2g+b+n-1, i<j \}.$
\item[\textbf{Relations:}]
\begin{multline*}
       \text{(PR1)}\ A_{i,j}^{-1} A_{r,s} A_{i,j} = A_{r,s}\ \mbox{ 
       if } \, \,(i<j<r<s) \; \mbox{ or }\ (r+1<i<j<s),\\
\shoveright{\mbox{or}\ (i=r+1<j<s\ \mbox{ for even }\ r<2g\ \mbox{ or }\ 
r>2g )\,;}
\\
\shoveleft{\text{(PR2)}\  A_{i,j}^{-1} A_{j,s} A_{i,j} = A_{i,s} 
A_{j,s} A_{i,s}^{-1}\ \mbox{ if }\ (i<j<s)\,;}\\
\shoveleft{\text{(PR3)}\ A_{i,j}^{-1} A_{i,s} A_{i,j} = A_{i,s} 
A_{j,s} A_{i,s} A_{j,s}^{-1} A_{i,s}^{-1}\ \mbox{ if }\ (i<j<s)\,;}\\
\shoveleft{\text{(PR4)}\ A_{i,j}^{-1}A_{r,s}A_{i,j}=A_{i,s}A_{j,s}A_{i,s}^{-1}A_{j,s}^{-1}A_{r,s}
A_{j,s}A_{i,s}A_{j,s}^{-1}A_{i,s}^{-1}}\ \mbox{ if }\ (i+1<r<j<s)\\
\shoveright{\mbox{or }\ (i+1=r<j<s\ \mbox{ for odd }\ r<2g\ \mbox{ or }\ r>2g )\, ;}\\
\shoveleft{\text{(ER1)}\ A_{r+1,j}^{-1} A_{r,s} A_{r+1,j}=A_{r,s} 
A_{r+1,s} A_{j,s}^{-1} A_{r+1,s}^{-1}\ \mbox{ if }\ r\mbox{ odd and }\ r<2g \, ;} \\
\shoveleft{\text{(ER2)}\ A_{r-1,j}^{-1} A_{r,s} A_{r-1,j}= A_{r-1,s} A_{j,s} A_{r-1,s}^{-1} A_{r,s} A_{j,s}
A_{r-1,s} A_{j,s}^{-1}A_{r-1,s}^{-1}}\\
\shoveright{\mbox{if }\ r\mbox{ even and }\ r<2g\, .} \\
\end{multline*}
\end{enumerate}
\end{thm}

As a representative of the generator $A_{i,j}$, we may take a
geometric braid whose only non-trivial (non-vertical) strand is the $(j-2g-b+1)$th one. In 
Figure~\ref{generateur}, we illustrate the projection of such braids on the surface $\Sigma_{g,1}$. 
\begin{figure}[h]
\begin{center}
\psfig{figure=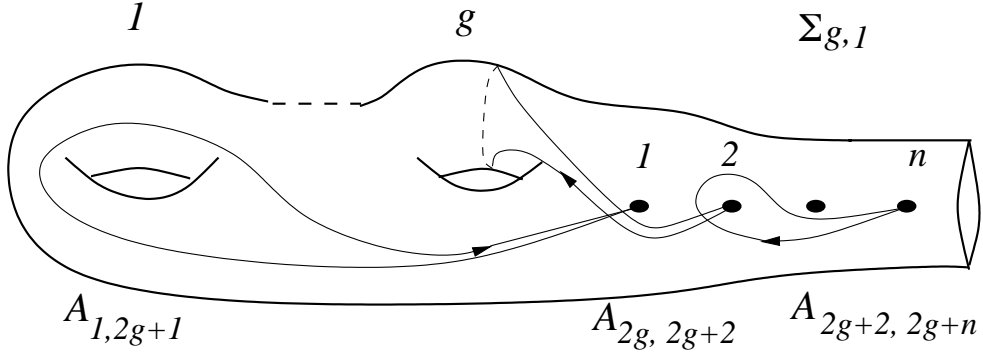,width=13cm}
\caption{\label{generateur}
Projection of representatives of the generators $A_{i,j}$. We
represent $A_{i,j}$ by its only non-trivial strand.}
\end{center}
\end{figure}

Recall that pure (framed) braid groups can be seen as subgroups of 
mapping class groups (see Proposition~\ref{prop:braids as mapping 
classes} and~\ref{prop:framed braids as mapping classes}). The 
isomorphism $P_{n}(\Sigma_{g,b})\approx\ker{\varphi_{n}}$ (where 
$\varphi_{n}$ forgets the punctures) is defined as follows: to an 
element $h$ of $\PPMCGB$ which is isotopic to the identity in 
$\Sigma_{g,b}$ (\emph{i.e.} $h\in\ker{\varphi_{n}}$), we associate 
the braid $t\mapsto\bigl(H_{t}(p_{i})\bigr)_{1\leq i\leq n}$ where 
$H:\Sigma_{g,b}\times I\to\Sigma_{g,b}$ is an isotopy between 
$\mathrm{Id}$ and $h$. From this point of view, one can easily see 
that the $A_{i,j}$'s correspond to the following elements of $\PPMCGB$:
$$
\begin{cases}
\tau_{\beta_{r}}\tau_{\beta_{r,s}}^{-1} & \textrm{if }\ i=2r-1, 1\leq r\leq g,\ j=2g+b+s-1,\ 1\leq s\leq n,\\
\tau_{\alpha_{r}}\tau_{\alpha_{r,s}}^{-1} & \textrm{if }\ i=2r, 1\leq r\leq g,\ j=2g+b+s-1,\ 1\leq s\leq n,\\
\tau_{\delta_{r}}\tau_{\delta_{r,s}}^{-1} & \textrm{if }\ i=2g+r, 1\leq r\leq b-1,\ j=2g+b+s-1,\ 1\leq s\leq
n,\\
\tau_{\delta_{r,s}}^{-1} & \textrm{if }\ i=2g+b+r-1,\ j=2g+b+s-1,\ 1\leq r<s\leq n,
\end{cases}
$$
where curves are those described by figure~\ref{courbes-1}.

\bigskip
Now, let us loot at the closed case. The group $P_{n}(\Sigma_{g})$ is the quotient of 
$P_{n}(\Sigma_{g,1})$ by the following relations (see~\cite{B}):

\setlength{\multlinegap}{0mm}
$$
\text{(TR)}\ [A_{2g,2g+k}^{-1},A_{2g-1,2g+k}]\cdots[A_{2,2g+k}^{-1},A_{1,2g+k}]=$$
$$=A_{2g+1,2g+k}\cdots A_{2g+k-1,2g+k} A_{2g+k,2g+k+1}\cdots A_{2g+k,2g+n}$$
 $$(1\leq k\leq n, \; \mbox{with the notation} \;  A_{2g+1,2g+1}=A_{2g+n,2g+n}=1)
$$

\begin{figure}[h]
\begin{center}
\psfig{figure=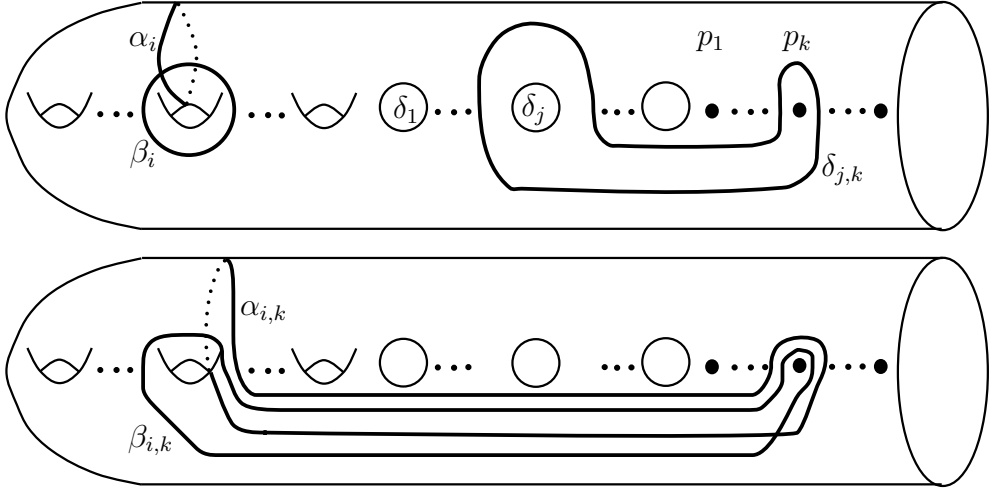}
\caption{\label{courbes-1}curves on $\Sigma_{g,b}$}
\end{center}
\end{figure}  

From this result, we can prove the following:
\begin{thm} \label{presentationframed}
Let $\Sigma_{g}$ be a compact, connected, closed, orientable surface 
of genus $g\ge\nolinebreak[4] 2$. The framed pure braid group $FP_n(\Sigma_{g})$ admits the following presentation:
\begin{enumerate}
\item[\textbf{Generators:}] $\{B_{i,j},\ f_{k}\; | \;1 \le i \le 
2g+n-1, 2g+1 \le j \le 2g+n,   \; i<j,\ 1\leq k\leq n \}.$
\item[\textbf{Relations:}] relations (PR1-4) and (ER1-2) together with the following:
\begin{eqnarray*}
&\text{(C)}&\text{ the } f_{k}\text{'s are central};\\
&&\\
&\text{(FTR)}& [B_{2g,2g+k}^{-1},B_{2g-1,2g+k}]\cdots[B_{2,2g+k}^{-1},B_{1,2g+k}]=\\
&&B_{2g+1,2g+k}\cdots B_{2g+k-1,2g+k}B_{2g+k,2g+k+1}\cdots B_{2g+k,2g+n}\,f_{k}^{2(g-1)}\\
\\
&& (1\leq k\leq n, \; \mbox{with the notation} \;  B_{2g+1,2g+1}=B_{2g+n,2g+n}=1)
\end{eqnarray*}
\end{enumerate}

\end{thm}

\begin{proof}
Consider the sequence~\eqref{eq:pure framed braid and pure braid}:
\begin{equation*}
1\to \Z^{n}\to FP_{n}(\Sigma_{g})\to P_{n}(\Sigma_{g})\to 1.
\end{equation*}
In terms of mapping class groups, $\Z^{n}$ is generated by $\tau_{\delta_{1}},\ldots,\tau_{\delta_{n}}$ where 
$\delta_{k}$ is a curve parallel to the $k^{\text{th}}$-boundary component. As shown in~\cite{J}, a 
presentation of $FP_{n}(\Sigma_{g,b})$ can be established as follow. Take as generators 
$$\{\tau_{\delta_{1}},\ldots,\tau_{\delta_{n}}\}\cup\{B_{i,j},\ 1 \le 
i \le 2g+n-1, 2g +1\le j \le 2g+n, i<j\}$$
where $B_{i,j}$ is a representative of $A_{i,j}$. Relations are of three types: the first corresponds 
to relations between the $\tau_{\delta_{k}}$'s, the second to lifting of each relations in 
$P_{n}(\Sigma_{g})$. The last one comes from the action under conjugation of each $B_{i,j}$ on the 
$\tau_{\delta_{k}}$'s. In order to define the $B_{i,j}$'s, consider the curves in 
Figure~\ref{courbes-1} where the boundary is capped by a disk and the marked points are replaced by holes. Then, we put

\begin{equation}\label{eq:generators of FPn}
B_{i,j}=\begin{cases}
\tau_{\beta_{r}}\tau_{\beta_{r,s}}^{-1}\tau_{\delta_{s}} & \textrm{if }\ i=2r-1, 1\leq r\leq g,\ j=2g+s,\
1\leq s\leq n,\\
\tau_{\alpha_{r}}\tau_{\alpha_{r,s}}^{-1}\tau_{\delta_{s}} & \textrm{if }\ i=2r, 1\leq r\leq g,\ j=2g+s,\
1\leq s\leq n,\\
\tau_{\delta_{r}}\tau_{\delta_{r,s}}^{-1}\tau_{\delta_{s}} & \textrm{if }\ i=2g+r,\ j=2g+s,\ 1\leq
r<s\leq n.
\end{cases}
\end{equation}

When capping each boundary components of $\Sigma_{g,n}$, the $\tau_{\delta_{k}}$'s are sent to identity, thus 
$B_{i,j}$ is indeed a representative of $A_{i,j}$.

Now, the $\tau_{\delta_{k}}$'s are central in $\Gamma_{g,n}$, thus the first and third type of relations 
correspond clearly to relations (C). Let us look at relations (PR1-4), (ER1-2) and (TR) of
$P_{n}(\Sigma_{g})$.
Relations  (PR1-4) -- (ER1-2) have been verified in~\cite[Lemma~19]{BGG} to hold in $\Gamma_{g,n}$ and
therefore in $FP_{n}(\Sigma_{g})$. Then it suffices to lift    (TR) relation from  $P_{n}(\Sigma_{g})$ to
$FP_{n}(\Sigma_{g})$.

\begin{figure}
\begin{center}
\psfig{figure=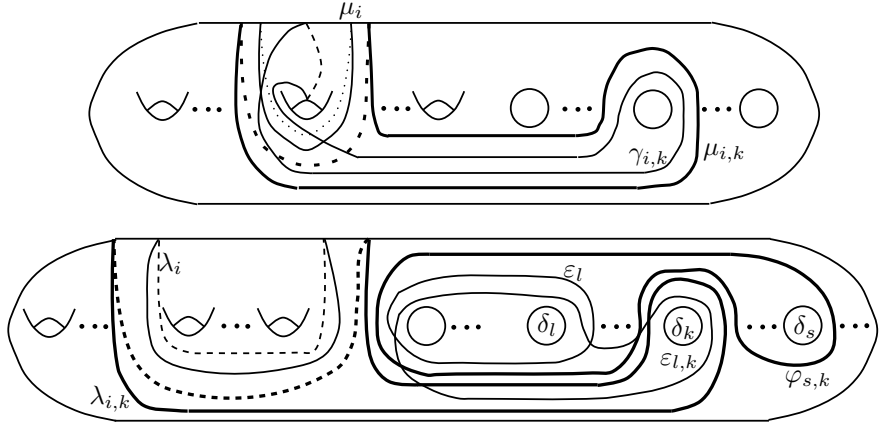}
\caption{\label{courbes-2}curves on $\Sigma_{g}$}
\end{center}
\end{figure}  

\bigskip\noindent
Since $\tau_{\beta_{i,k}}\tau_{\beta_{i}}^{-1}(\alpha_{i})=\gamma_{i,k}$ and 
$\tau_{\beta_{i,k}}\tau_{\beta_{i}}^{-1}(\alpha_{i,k})=\alpha_{i}$ where $\gamma_{i,k}$ is the curve 
descibed in figure~\ref{courbes-2}, one has (we omit the $\tau_{\delta_{l}}'s$ because they are 
central)\footnote{Recall that we denote composition of applications from left to right. Thus, if 
$\gamma=h(\alpha)$, one has $\tau_{\gamma}=h^{-1}\tau_{\alpha}h$.}:
$$
\begin{array}{r@{\,=\,}l}
[B_{2i,2g+k}^{-1},B_{2i-1,2g+k}] & 
\tau_{\alpha_{i}}^{-1}\tau_{\alpha_{i,k}}\tau_{\beta_{i}}\tau_{\beta_{i,k}}^{-1}
\tau_{\alpha_{i}}\tau_{\alpha_{i,k}}^{-1}\tau_{\beta_{i}}^{-1}\tau_{\beta_{i,k}}\\
& \tau_{\alpha_{i}}^{-1}\tau_{\alpha_{i,k}}\tau_{\gamma_{i,k}}\tau_{\alpha_{i}}^{-1}.
\end{array}
$$

Considering the curves described in figure~\ref{courbes-2}, one has the following lantern 
relations: 
$$\tau_{\alpha_{i}}^{2}\tau_{\mu_{i,k}}\tau_{\delta_{k}}=\tau_{\mu_{i}}\tau_{\alpha_{i,k}}\tau_{\gamma_{i,k}}$$
and we get
$$[B_{2i,2g+k}^{-1},B_{2i-1,2g+k}]=\tau_{\mu_{i}}^{-1}\tau_{\mu_{i,k}}\tau_{\delta_{k}}.$$
Then, using lantern relations
$$\tau_{\mu_{i-1}}\tau_{\lambda_{i}}\tau_{\lambda_{i-1,k}}\tau_{\delta_{k}}=\tau_{\lambda_{i-1}}
\tau_{\lambda_{i,k}}\tau_{\mu_{i-1,k}},$$
it is easy to check by induction on $i$ that\footnote{Note that $\lambda_{g}=\mu_{g}$ and 
$\lambda_{g,k}=\mu_{g,k}$.}
\begin{equation}\label{left member}
[B_{2g,2g+k}^{-1},B_{2g-1,2g+k}]\cdots[B_{2i,2g+k}^{-1},B_{2i-1,2g+k}]=
\tau_{\lambda_{i}}^{-1}\tau_{\lambda_{i,k}}\tau_{\delta_{k}}^{2(g-i)+1}.
\end{equation}
On the other hand, by definition of the $B_{i,j}$'s, on has, for $1\leq k\leq n$:
\setlength{\multlinegap}{0mm}
\begin{multline*}
B_{2g+1,2g+k}\cdots B_{2g+k-1,2g+k}B_{2g+k,2g+k+1}\cdots B_{2g+k,2g+n}=\\
\tau_{\delta_{1}}\tau_{\delta_{1,k}}^{-1}\cdots\tau_{\delta_{k-1}}\tau_{\delta_{k-1,k}}^{-1}
\tau_{\delta_{k+1}}\tau_{\delta_{k,k+1}}^{-1}\cdots\tau_{\delta_{n}}\tau_{\delta_{k,n}}^{-1}\tau_{\delta_{k}}^{n-1}
\end{multline*}

\medskip\noindent
Now, with the curves described in figure~\ref{courbes-2}, one has the lantern relation
$$
\tau_{\varepsilon_{i-1}}\tau_{\delta_{i}}\tau_{\delta_{k}}\tau_{\varepsilon_{i,k}}=
\tau_{\varepsilon_{i}}\tau_{\delta_{i,k}}\tau_{\varepsilon_{i-1,k}},
$$
from which one can check by induction on $i< k$ that ($\varepsilon_{1}=\delta_{1}$ and $\varepsilon_{1,k}=\delta_{1,k}$) 
$$\tau_{\delta_{1,k}}^{-1}\tau_{\delta_{2,k}}^{-1}\cdots\tau_{\delta_{i,k}}^{-1}=
\tau_{\varepsilon_{i}}\tau_{\varepsilon_{i,k}}^{-1}\tau_{\delta_{1}}^{-1}\tau_{\delta_{2}}^{-1}
\cdots\tau_{\delta_{i}}^{-1}\tau_{\delta_{k}}^{1-i}.
$$
Then, with the lantern relation (where $s>k$ and $\varphi_{k,k}=\varepsilon_{k-1}$)
$$
\tau_{\varepsilon_{s}}\tau_{\delta_{s}}\tau_{\delta_{k}}\tau_{\varphi_{s-1,k}}=
\tau_{\varphi_{s,k}}\tau_{\delta_{k,s}}\tau_{\varepsilon_{s-1}},
$$
one obtains by induction on $s$
\begin{equation*}
\tau_{\delta_{1,k}}^{-1}\cdots\tau_{\delta_{k-1,k}}^{-1}\tau_{\delta_{k,k+1}}^{-1}\cdots\tau_{\delta_{k,s}}^{-1}=
\tau_{\varphi_{s,k}}\tau_{\varepsilon_{s}}^{-1}\tau_{\delta_{s}}^{-1}\cdots\tau_{\delta_{k+1}}^{-1}
\tau_{\delta_{1}}^{-1}\cdots\tau_{\delta_{k-1}}^{-1}\tau_{\delta_{k}}^{2-s}
\end{equation*}
and finally
\begin{align}
B_{2g+1,2g+k}\cdots B_{2g+k-1,2g+k}B_{2g+k,2g+k+1}\cdots B_{2g+k,2g+n} & =  
\tau_{\varphi_{n,k}}\tau_{\varepsilon_{n}}^{-1}\tau_{\delta_{k}}^{2-n}\tau_{\delta_{k}}^{n-1}\notag\\
& 
=\tau_{\varphi_{n,k}}\tau_{\varepsilon_{n}}^{-1}\tau_{\delta_{k}}\label{eq:right member}.
\end{align}
To conclude, one has just to compare equation~\eqref{left member} 
with $i=1$ and equation~\eqref{eq:right member}, 
 and to observe that $\varepsilon_{n}=\lambda_{1}$ and $\varphi_{n,k}=\lambda_{1,k}$.
\end{proof}

\begin{rem}
    This presentation gives another proof of the fact that the 
    central extension~\eqref{eq:pure framed braid and pure braid} does not 
    split for all $g \ge 2$. Indeed, if it splits, the group 
    $FP_n(\Sigma_{g})$ is the direct product of $\Z^{n}$ and 
    $P_n(\Sigma_{g})$: it contredicts the relation {\it (FTR)}.
\end{rem}


\section{Framed pure braids as centralizers of Dehn twists}  \label{section:frameddehn}

We give a third possible definition of framed pure braid groups in terms of centralizers of Dehn twists:
since in the case of surfaces with boundary the corresponding
framed pure braid groups are trivial central extension of pure braid groups, we will focus on the case of closed surfaces of genus greater than one.
This interpretation of framed pure braid groups will allow us to 
introduce another possible definition of framed braid groups 
(Section~\ref{section:framesbraidsas2nstrands}).
All following Definitions and Propositions can be easily extended to the case with boundary.

Let us consider a surface $\Sigma_{g,n}$ of genus $g\ge 2$ and $n$ boundary components and 
cap each boundary component  with a disk with $2$ marked points (see Figure~\ref{figure:framedasbraids}).
We will denote by  $\chi_n: \Gamma_{g,n}  \to  \mathrm{P}\Gamma^{2n}_{g}$
the morphism induced by the inclusion  of $\Sigma_{g,n}$ into 
$\Sigma_{g}^{2n}$, a surface of genus $g$ with $2n$ marked points.

One has the  following commutative diagram at the level of mapping classes:

\begin{figure} 
	\centering
	\includegraphics[width=12cm]{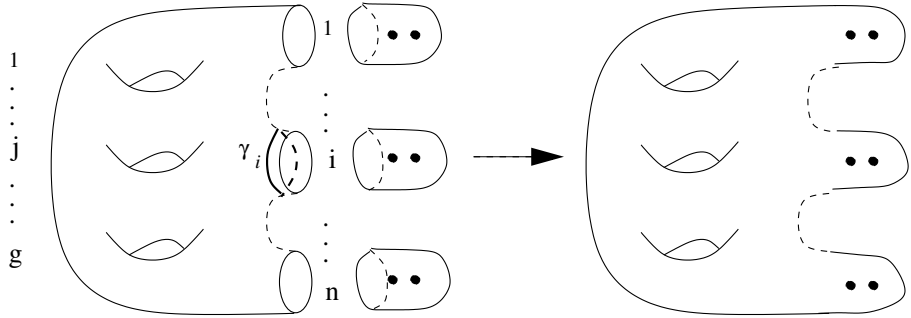}
	\caption{Inclusion of $\Sigma_{g,n}$ into $\Sigma_{g}^{2n}$\label{figure:framedasbraids}}
\end{figure}

$$
\xymatrix{
 1   \ar[r] &\ FP _n(\Sigma_g)  \ar[r] \ar[d]  & \Gamma_{g,n}   \ar[r]^{\lambda_g^n} \ar[d]^{\chi_n} & \Gamma_g \ar[d]^{id}  \ar[r] &1\\
 1   \ar[r] &  P_{2n}(\Sigma_g) \ar[r] &  \mathrm{P}\Gamma^{2n}_{g} \ar[r]^{\phi_{2n}}  & \Gamma_g \ar[r] & 1
  } 
$$
Since   $\chi_n$ is injective (see Proposition 4.1 in~\cite{PR1}), we 
can consider  $FP_n(\Sigma_g) $ as a subgroup of  $P_{2n} (\Sigma_g)$.

Moreover, we can provide a characterization of $FP_n(\Sigma_g) $ as subgroup of $\mathrm{P}\Gamma^{2n}_{g}$.

\begin{prop}\label{prop:thirdequiv}
Let  $g\ge 2$, $\chi_n: \Gamma_{g,n} \to \mathrm{P}\Gamma_g^{2n}$ be the map defined above and  $\gamma_i$ be the boundary
of the $i$th capping punctured disk (see Figure~\ref{figure:framedasbraids}).
The group  $\chi_n(FP_n(\Sigma_g))$ coincides with ${\displaystyle 
\build{\cap}{j=1}{n}C_{P_{2n} (\Sigma_g)}(\tau_{\gamma_j})}$,
where $C_{P_{2n} (\Sigma_g)}(\tau_{\gamma_j})$ is the centralizer of the Dehn twist  $\tau_{\gamma_j}$ in $P_{2n} (\Sigma_g)$.
\end{prop}

Before proving~\repr{thirdequiv} we need to introduce some definitions and to recall some results.
Let $\Sigma$ be a surface with a finite set  $\mathcal{P}$ of $n$ 
marked points. An arc is an embedding $A:[0,1] \to \Sigma$ such that 
$A(0), A(1) $ are in $\mathcal{P}$ and $A(x)$ is in $ \Sigma 
\setminus \mathcal{P}$ for all $x $ in $ (0,1)$. A $(j,k)$-arc is  an 
arc such that  $A(0)=j$ and $A(1)=k$. Note that two $(j,k)$-arcs  are 
isotopic if and only if they can be connected by a continuous family of $(j,k)$-arcs.

For any  $i=1, \dots, n $ we fix   an arc with end points  $2i-1$ and  $2i$ in the interior of the 
$i$th capping disk. In a disk with $2$ punctures all arcs are isotopic; let us choose a representative 
as in Figure~\ref{figure:figurearc} that   we will denote  by $[2i-1, 2i]$.

\begin{figure}
	\centering
\includegraphics[width=5cm]{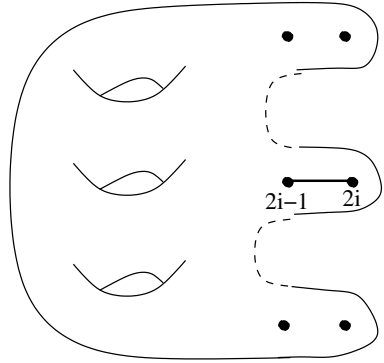}
\caption{\label{figure:figurearc} The arc $[2i-1, 2i]$.}
\end{figure}

 Consider a disk $\D_{2n}$ in $\Sigma_g^{2n} $ which contains all $2n$ punctures. The embedding of $\D_{2n}$  into $\Sigma_g^{2n}$
 induces an embedding of  $B_{2n}$ into  $B_{2n}(\Sigma_g)$ (see for instance~\cite{PR1}):
the usual generator $\sigma_j$ of  $B_{2n}$ can be considered therefore as an element of $B_{2n}(\Sigma_g)$;
in particular the generator $\sigma_{2j-1}$ for $j=1,\ldots,n$
 corresponds to  the \emph{braid twist}  (see~\cite{LP} for a definition)  associated to the arc $[2j-1, 2j]$.

The group  $B_{2n}(\Sigma_g)$  acts on $\Sigma_g$ up to  isotopy.
In the following we adopt the convention that, for any $\beta$ in $B_{2n}(\Sigma_g)$,
$*\beta: \Sigma_g \to \Sigma_g$ corresponds to a mapping  $\Sigma_g \times \{0\} \to \Sigma_g \times \{1\}$
and defines an action on the right, whereas
$\beta *: \Sigma_g \to \Sigma_g$ corresponds to a mapping  $\Sigma_g \times \{1\} \to \Sigma_g
\times \{0\}$ and defines an action on the left.

In particular   $B_{2n}(\Sigma_g)$ acts on the right and on the left on
$\mathcal{A}_{2n}$, the set of arcs up to isotopy.

\begin{thm} ~\cite{belcen} \label{centr}
Let $g>1$. For each $\beta$ in $B_{2n}(\Sigma_g)$ the following properties are equivalent:
\begin{enumerate}
        \item $\sigma_{2j-1} \beta = \beta \sigma_{2k-1} \, $, 
        \item $\sigma_{2j-1}^r \beta = \beta \sigma_{2k-1}^r \, $ for any integer $r$,  
	\item  $\sigma_{2j-1}^r \beta = \beta \sigma_{2k-1}^r \,$ for some nonzero  integer $r$,  
	\item $[2j-1,2j]*\beta=[2k-1,2k]$.
\end{enumerate}
\end{thm}

\begin{rem}
Theorem~\ref{centr} is a weaker version of Theorem~2.2 in~\cite{belcen}, which statement concerns
all braid groups  $B_n(\Sigma_g)$ and all braid generators $\sigma_1, \ldots, \sigma_{n-1}$.
\end{rem}

We recall that $\gamma_i$ denotes the boundary of the $i$th capping disk.
From the fact that  $\sigma_{2i-1,2i}^2= \tau_{\gamma_i}$ we can therefore deduce the following result:

\begin{cor}\label{centrcor}
Let $g>1$. For each $\beta$ in $P_{2n}(\Sigma_g)$ the following properties are equivalent:
\begin{enumerate}
        \item $\tau_{\gamma_i} \beta = \beta \tau_{\gamma_i} \, $, 
        \item  $\tau_{\gamma_i}^r \beta = \beta \tau_{\gamma_i}^r  , \; $  for any integer $r$,  
	\item $\tau_{\gamma_i}^r \beta = \beta  \tau_{\gamma_i}^r , \;$ for some nonzero  integer $r$,  
	\item $[2i-1,2i]*\beta=[2i-1,2i]$.
\end{enumerate}
\end{cor}

\begin{rem}\label{centrrem} The equivalences between a) and d) are in the folkrore even for the case $g=1$.
\end{rem}

\begin{prooftext}{\repr{thirdequiv}}
We recall that a multitwist is a product of twists along pairwise disjoint curves
and that the map   $\chi_n:  \Gamma_{g,n} \to   \mathrm{P}\Gamma^{2n}_{g}$ defined above is injective
and sends $FP_n(\Sigma_g)$ into $P_{2n} (\Sigma_g)$.
Any  generator $g$ of $FP_n(\Sigma_g)$ is a multitwist 
(see~\eqref{eq:generators of FPn}), as well as its image
$\chi_n(g)$. Since two twists $\tau_\gamma$ and $\tau_\delta$ commute if and only if $\gamma$ and $\delta$ are 
disjoint up to isotopy (see for instance~\cite{PR2}), we deduce that any element  $\chi_n(g)$ commutes with
$\tau_{\gamma_j}$ for $j=1,\ldots, n$ and therefore 
$\chi_n(FP_n(\Sigma_g))$ is a subgroup of $\build{\cap}{j=1}{n}C_{P_{2n} (\Sigma_g)}(\tau_{\gamma_j})$.
On the other hand, if $g$ is an element of $\build{\cap}{j=1}{n}C_{P_{2n} (\Sigma_g)}(\tau_{\gamma_j})$, from Corollary~\ref{centrcor}
it follows that $[2j-1,2j] * g = [2j-1,2j] $ for $j=1,\ldots, n$ and then up to isotopy we can suppose that
$g=Id$ on each capping disk.
Therefore we can consider $g$ as the image by $\chi_n$ of an element 
$\tilde{g}$ in $\Gamma_{g,n}$. Since
$\lambda_g^n(\tilde{g})=\phi_{2n}(\chi_n(\tilde{g}))=\phi_{2n}(g)=1$ we deduce that   $\tilde{g} \in FP_n(\Sigma_g)$ and hence that
$FP_n(\Sigma_g)$ maps onto $\build{\cap}{j=1}{n}C_{P_{2n} (\Sigma_g)}(\tau_{\gamma_j})$.

\end{prooftext}

Let  $A_{k,l}$ be  the generator of $P_{2n}(\Sigma_g)$ depicted in~\reth{purepres} for   $1 \le k \le 2g+2n-1$
and $2g +1\le l \le 2g+2n $.
Now, set:
\setlength{\multlinegap}{8mm}
\begin{multline*}
C_{i,j} = A_{i,2(j-g)-1}A_{i,2(j-g)}A_{2(j-g)-1,2(j-g)}\ \text{ for 
}\ i=1,\dots, 2g\ \text{ and }\\  j=2g+1, \ldots, 2g+n,
\end{multline*}
\vspace*{-10mm}\setlength{\multlinegap}{0mm}
\begin{multline*}
C_{i,j}=A_{2(i-g)-1,2(j-g)-1}A_{2(i-g),2(j-g)-1}A_{2(i-g)-1,2(j-g)}A_{2(i-g),2(j-g)}A_{2(j-g)-1,2(j-g)}
\\ \text{ for }\ 2g+1\le i <j\le 2g+n
\end{multline*}
and finally   $F_{k}= A_{2g+2k-1, 2g+2k}$ for $k=1,\ldots,n$.

Roughly speaking, the element $F_{k}$ corresponds to $\sigma_{2k-1,2k}^2$ 
in $P_{2n}(\Sigma_g)$
and the $C_{i,j}$'s correspond  to pure braids on $P_{2n}(\Sigma_g)$
where the only non trivial strands are the $(2(j-g)-1)$th and the  $2(j-g)$th one, which are ``parallel'',
meaning that they bound an annulus in $\Sigma_g \setminus \{p_{1}, \ldots, p_{2(j-g-1)}, p_{2(j-g)+1}, \dots, p_{2n} \, \}$.

We recall that it is possible to embed $P_n$ into $P_{2n}$ (or $B_n$ into $B_{2n}$) ``doubling'' any  strand (see for instance
\cite{FRZ}); one can remark that in the case of braid groups on closed surfaces of genus $g\ge2$ such embeddings are not well defined because of 
\reth{extension centrale par Zn cas pure}.

From Theorem~\ref{presentationframed} and~\repr{thirdequiv} one can therefore deduce the following group presentation for
$FP_n(\Sigma_{g})$ as subgroup of $FP_{2n}(\Sigma_{g})$.

\begin{prop}\label{presentationpureframedascollectionsofpaths}
Let $\Sigma_{g}$ be a compact, connected, closed, orientable surface of genus $g > 1$. The framed braid group 
$FP_n(\Sigma_{g})$ as subgroup of $P_{2n}(\Sigma_{g})$ admits the following presentation:
\begin{enumerate}
\item[\textbf{Generators:}] $\{C_{i,j}, F_{k}\; | \;1 \le i \le 2g+n-1, 2g +1\le j \le 2g+n, $

\hfill $i<j,\ 1\leq k\leq n \}.$
\item[\textbf{Relations:}] relations (PR1-4) and (ER1-2) from~\reth{purepres} replacing $A_{i,j}$ with $C_{i,j}$ together with the following:
\begin{eqnarray*}
&\text{(C)}&\text{ the } F_{k}\text{'s are central};\\
&&\\
&\text{(FTR)}& [C_{2g,2g+k}^{-1},C_{2g-1,2g+k}]\cdots[C_{2,2g+k}^{-1},C_{1,2g+k}]=\\
&&C_{2g+1,2g+k}\cdots C_{2g+k-1,2g+k}C_{2g+k,2g+k+1}\cdots C_{2g+k,2g+n}\, F_{k}^{2(g-1)}\\
\\
&& (1\leq k\leq n, \; \mbox{with the notation} \;  C_{2g+1,2g+1}=C_{2g+n,2g+n}=1)
\end{eqnarray*}
\end{enumerate}
\end{prop}

\begin{proof}
 One has to check that $\chi_n$ sends the generators $B_{i,j}$ and $f_k$ respectively into $C_{i,j}$ and $F_{k}$.
\end{proof}


\section{Framed braids as $2n$-strands braids} \label{section:framesbraidsas2nstrands}

The definition of framed pure braid groups in terms of centralizers of Dehn twists given in~\repr{thirdequiv} allows us to give another equivalent definition
for the framed braid group of a closed surface introduced in Definition~\ref{defframed}.

\begin{deft}
Let $\Sigma$ be a surface of genus $g>1$ with a finite set  $\mathcal{P}$ of $2n$ marked points and let $\mathcal{I}_n$
be the set of arcs $\{[1,2], [3,4], \ldots,[2n-1,2n]\}$. We say that 
an element $\beta$ of $B_{2n}(\Sigma_g)$ preserves with orientation $\mathcal{I}_n$
if for any $[2i-1,2i]$ there exists $j$ such that $[2i-1,2i] * \beta = [2j-1,2j]$ with $\{2i-1\} * \beta = \{2j-1\}$ (and therefore $\{2i\} *  \beta =\{2j\}$).
The  set of braids preserving with orientation  $\mathcal{I}_n$ forms a subgroup of $B_{2n}(\Sigma_g)$  that we will denote by $\widetilde{FB}_{n}(\Sigma_g)$.
\end{deft} 

Let us recall that the group $B_{2n}(\Sigma_g)$
is generated (see~\cite{B}) by the usual generators $\sigma_1, \ldots, \sigma_{2n-1}$ of $B_{2n}$ plus the braids
$a_1, b_1, \ldots, a_g, b_g$ which are the pure braids $A_{1,2g+1},A_{2,2g+1},\ldots,A_{2g,2g+1}$ depicted in~\reth{purepres}.
Let us define 
$$A_i=A_{2i-1,2g+1}A_{2i-1,2g+2}A_{2g+1,2g+2}\ \text{ and }\ 
B_i=A_{2i,2g+1}A_{2i,2g+2}A_{2g+1,2g+2}$$
for $i=1,\ldots,g$ and let $\tau_j = \sigma_{2j}\sigma_{2j-1}\sigma_{2j+1}\sigma_{2j}$ for 
$j=1,\ldots,n-1$. The elements $A_1, B_1, \ldots, A_g, B_g,$   $ 
\tau_1, \ldots, \tau_{n-1},$   $F_{1}, \ldots, F_{n}$,
belong to $\widetilde{FB}_n(\Sigma_{g})$. Moreover we have the following result:

\begin{thm} \label{presentationbraidframed}
Let $\Sigma_{g}$ be a compact, connected, closed, orientable surface of genus $g > 1$. The  group 
$\widetilde{FB}_n(\Sigma_{g})$ admits the following presentation:
\begin{enumerate}
\item[\textbf{Generators:}] $A_1, B_1, \ldots, A_g, B_g, \tau_1, 
\ldots, \tau_{n-1}, F_{1}, \ldots, F_{n}$.
\item[\textbf{Relations:}] 
\begin{gather}
\text{$\tau_i F_j=F_j \tau_i$ for $j\not= i, i+1$}\\
\text{$\tau_iF_i=F_{i+1}\tau_i$}\\
\text{$\tau_iF_{i+1}=F_i\tau_i$}\\
\text{$\tau_i\tau_j=\tau_j\tau_i$ if $\lvert i-j \rvert \geq
2$}\label{eq:artin1}\\
\text{$\tau_i\tau_{i+1}\tau_i= \tau_{i+1}\tau_i \tau_{i+1}$
for all $1\leq i\leq n-2$}\label{eq:artin2}\\
\text{$c_i\tau_j= \tau_j c_i$ for all $j\geq 2$,   $c_i=A_i$ or
  $B_i$ and $i=1, \ldots, g$}\label{eq:asjg}\\
\text{$c_i \tau_1 c_i \tau_1= \tau_1 c_i \tau_1 c_i$   for
  $c_i=A_i$ or $B_i$ and $i=1, \ldots, g$}\label{eq:bsbg}\\
\text{$A_i \tau_1 B_i = \tau_1 B_i \tau_1 A_i \tau_1$ for  $i=1, \ldots, g$}\label{eq:abbag}\\
\text{$c_i \tau_1^{-1}  c_j \tau_1=\tau_1^{-1} c_j \tau_1 c_i$
  for $c_i=A_i$ or $B_i$, $c_j=A_j$ or $B_j$ and $1\le j<i\le g$}\label{eq:cddcg}\\
\text{$\prod_{i=1}^g [A_i^{-1},B_i]= \tau_1\cdots \tau_{n-2} \tau_{n-1}^2 \tau_{n-2} \cdots \tau_1 F_1^{2(g-1)}$}\label{eq:abcommg}.
\end{gather}
\end{enumerate}
\end{thm}
\begin{proof}
If $\pi: \widetilde{FB}_n(\Sigma_{g}) \to \sn$ is the map which associates to any $\beta\in \widetilde{FB}_n(\Sigma_{g}) $
the corresponding permutation on the set $\mathcal{I}_n$, we have the following exact sequence:
$$1\to FP_n(\Sigma_{g}) \to \widetilde{FB}_n(\Sigma_{g}) \to \sn \to 1.$$
Then, the statement follows by
Proposition~\ref{presentationpureframedascollectionsofpaths} and, as already done in~Theorem~\ref{presentationframed}, by another 
application of the technique
from~\cite{J} that we leave to the reader; we simply remark that
$ \tau_1, \ldots, \tau_{n-1}$ are coset representative of the usual generators of $\sn$ and the other generators of $FP_{n}(\Sigma_g)$ can be
deleted since they are conjugated by  words in  $ \tau_1, \ldots, \tau_{n-1}$.

\end{proof}

Before proving that $ \widetilde{FB}_n(\Sigma_{g}) $ is isomorphic to the framed braid group  $FB_n(\Sigma_{g})$ defined in Definition~\ref{defframed} 
let us recall few additional notations and results from~\cite{belcen}.

\begin{deft} \label{def1}
 We define a \emph{ribbon} as an embedding
             $$
             R: [0,1] \times [0,1] \to \Sigma \times [0,1] \, ,
             $$
             such that $R(s,t) $ is in $ \Sigma \times \{t\}$. 

Let  $A$ be a  $(j,k)$-arc in $\Sigma \times  \{0\}$. Then the isotopy corresponding
        to $\beta \in B_n(\Sigma_g)$ moves $A$ through a ribbon which is \emph{proper} for $\beta$, meaning that
        \begin{itemize}
               \item $R(0,t)$ and $R(1,t)$ trace out the strands $j$ and $k$ of the
                 braid $\beta$, while the rest of the ribbon is disjoint from $\beta$;
               \item $R([0,1]\times \{0\})=A$    and $R([0,1]\times \{1\})=A*\beta$.
        \end{itemize}
\end{deft}

\begin{deft} \label{def2}
      We say that the braid $\beta $ in $B_{2n}(\Sigma)$ has a $(2j-1,2k-1)$-band  if there exists a
      ribbon proper for $\beta$ and connecting $[2j-1,2j] \times \{0\}$ to $[2k-1,2k] \times\{1\}$.
\end{deft}

 \begin{prop} \cite[Proposition 2.2]{belcen} \label{prop:braidcaracter}
  Let $g>1$ and $1\le j<k\le n$. For each $\beta$ in $B_{2n}(\Sigma_g)$, the following properties are equivalent:
      \begin{enumerate}
        \item $\beta$ has a $(2j-1,2k-1)$-band, 
        \item $[2j-1,2j]*\beta=[2k-1,2k]  $.
      \end{enumerate}
  \end{prop}

\begin{prop} \label{framedbraidsascollectionofpaths}
The group  $ \widetilde{FB}_n(\Sigma_{g}) $ is isomorphic to the 
framed braid group\linebreak[4] $FB_n(\Sigma_{g})$ defined in Definition~\ref{defframed}.
\end{prop}
\begin{proof}
By its definition an element $\beta \in FB_n(\Sigma_{g})$ can be seen 
as an element\linebreak[4] $\beta' \in B_{2n}(\Sigma_{g})$ such that for 
any $i=1,\ldots,n$ there exists a $(2i-1,2k-1)$-band and   $\{2i-1\}* \beta' = \{2k-1\}$.
This defines a morphism from $FB_n(\Sigma_{g})$ to 
$\widetilde{FB}_n(\Sigma_{g})$ wich is an isomorphism by~\repr{braidcaracter}.
\end{proof}


\section{Framed braids on the torus}

In the case of the torus the proposed definitions are not equivalent and let arise different notions of framings.

Let  $F_{n}(\Sigma)$ and $\widehat{F_{n}}(\Sigma)$ be the spaces defined at the beginning of the Section \ref{section:definitions}. We recall 
Definition \ref{defframed} in the case of $\Sigma=\T$.

\begin{deft}\label{defframed2}
The \emph{pure framed braid group} $FP_{n}(\T)$ on $n$ strands of $\T$ is the fundamental group of 
$F_{n}(\T)$. The \emph{framed braid group} on $n$ strands of $\T$ is the fundamental group of 
$\widehat{F_{n}}(\T)$.
\end{deft}

\begin{deft}\label{def:capping}
We denote by $\widetilde{FP}_n(\T)$ the kernel of the morphism 
$\lambda_{1}^{n}:\Gamma_{1,n}\to\Gamma_{1}$ induced by capping $n$ boundary components.
\end{deft}

In what follows we prove that $\widetilde{FP}_n(\T)$ is a quotient of $FP_n(\T)$
and we examine exact sequences for $FP_n(\T)$ and $\widetilde{FP}_n(\T)$.

\begin{thm} \label{th:extension centrale par Zn cas pure tore}
The group $FP_n(\T)$ is isomorphic to $\Z^{n}\times P_{n}(\T)$
and  $F_{n}(\T)$ is an Eilenberg-Maclane space of type $(FP_{n}(\T),1)$.
\end{thm}
\begin{proof}
Since  $\T$ is parallelizable, the proof of Theorem \ref{th:extension centrale par Zn cas pure tore} 
is the same as in the case of surfaces with boundary given  in Theorem \ref{th:extension centrale par Zn cas pure}.
\end{proof}

\begin{thm}\label{th:framedsequence tore}
For $n\geq 0$ and $m\geq 0$, one has the following splitting exact sequence:
\begin{equation*}
(FSPB) \qquad 1\to FP_{m}\left(\T \setminus n \; \mbox{discs} \;\right)\to 
FP_{n+m}(\T)\stackrel{\alpha_{n,m}}{\to}FP_{n}(\T)\to 1
\end{equation*}
where $\alpha_{n,m}$ consists in forgetting the first $m$ strands.
\end{thm}
\begin{proof}
Because of Theorem \ref{th:extension centrale par Zn cas pure tore} 
the (FSPB) sequence reduces to  the exact sequence induced by Fadell-Neuwirth fibration
\begin{equation*}\label{eq:pure braid sequence2}
(SPB) \quad 1\to P_{m}\left(\T \setminus n \; \mbox{points} \;\right)\to P_{n+m}(\T)\to 
P_{n}(\T)\to 1
\end{equation*}
which is a splitting sequence (\cite{GG}).
\end{proof}

Let us denote $ \widetilde{P}_{n}(\T)$  the kernel of the morphism 
$\phi_{n}: \mathrm{P}\Gamma_{1}^n \to\Gamma_{1}$ induced by forgetting the $n$ marked points.
We recall that, according to \repr{braids as mapping classes}, one has that  $  \widetilde{P}_{n}(\T) \simeq  P_{n}(\T)/Z(P_{n}(\T))$.

\begin{thm} \label{th:purepres2}
 The group $\widetilde{P}_n(\T)$ admits the following
presentation:
\begin{enumerate}
\item[\textbf{Generators:}] $\{A_{i,j}\; | \;1 \le i \le n+1, 3\le j \le n+2, i<j \}.$
\item[\textbf{Relations:}]
\begin{multline*}
       \text{(PR1)}\ A_{i,j}^{-1} A_{r,s} A_{i,j} = A_{r,s}\ \mbox{if} \, \,(i<j<r<s) \; \mbox{or}\ (r+1<i<j<s),\\
\shoveright{\mbox{or}\ (i=r+1<j<s\ \mbox{ or }\  r>2 )\,;}\\
\shoveleft{\text{(PR2)}\  A_{i,j}^{-1} A_{j,s} A_{i,j} = A_{i,s} A_{j,s} A_{i,s}^{-1}\ \mbox{ if }\ (i<j<s)\,;}\\
\shoveleft{\text{(PR3)}\ A_{i,j}^{-1} A_{i,s} A_{i,j} = A_{i,s}  A_{j,s} A_{i,s} A_{j,s}^{-1} A_{i,s}^{-1}\ \mbox{ if }\ (i<j<s)\,;}\\
\shoveleft{\text{(PR4)}\ A_{i,j}^{-1}A_{r,s}A_{i,j}=A_{i,s}A_{j,s}A_{i,s}^{-1}A_{j,s}^{-1}A_{r,s} A_{j,s}A_{i,s}A_{j,s}^{-1}A_{i,s}^{-1}}\ \mbox{ if }\ (i+1<r<j<s)\\
\shoveright{\mbox{or }\ (i+1=r<j<s\ \mbox{ for odd }\  \mbox{ or }\ r>2 )\, ;}\\
\shoveleft{\text{(ER1)}\ A_{2,j}^{-1} A_{1,s} A_{2,j}=A_{1,s}  A_{2,s} A_{j,s}^{-1} A_{2,s}^{-1}}  \\
\shoveleft{\text{(ER2)}\ A_{1,j}^{-1} A_{2,s} A_{1,j}= A_{1,s} A_{j,s} A_{1,s}^{-1} A_{2,s} A_{j,s} A_{1,s} A_{j,s}^{-1}A_{1,s}^{-1}}\\
\shoveleft{\text{(TR)}\  [A_{2,2+k}^{-1},A_{1,2+k}]=A_{3,2+k}\cdots A_{1+k,2+k} A_{2+k,3+k}\cdots A_{2+k,2+n}}\  (1\leq k\leq n)\\
\shoveleft{(1\leq k\leq n, \; \mbox{with the notation} \;  A_{3,3}=A_{2+n,2+n}=1)}\\
\shoveleft{\text{(QR1)}\ A_{1,3}\cdots A_{1,2+n}=1}\\
\shoveleft{\text{(QR2)}\ A_{2,3}\cdots A_{2,2+n}=1}\\
\end{multline*}
\end{enumerate}
\end{thm}

\begin{proof}
It suffices to remark that the previous presentation is the presentation for the group $P_n(\T)$
given in \reth{purepres}, quotiented by 
(QR1) and (QR2) relations, where  $A_{1,3}\cdots A_{1,2+n}, A_{2,3}\cdots A_{2,2+n}$ generate the center of
$P_n(\T)$ (see for instance \cite{PR1}).
\end{proof}

One has the following commutative 
diagram with exact rows and columns:

$$
\xymatrix{%
&&1 \ar[d]  & 1 \ar[d]& \\
1 \ar[r] & \Z^{n} \ar[r]\ar@{=}[d] & \widetilde{FP}_{n}(\T) \ar[r] \ar[d] & \widetilde{P}_{n}(\T) \ar[r] \ar[d] &
1 \ \ \ \ \ \ \ \ \ 
(\ref{eq:pure framed braid and pure braid})\\
1 \ar[r] & \Z^{n} \ar[r] &\Gamma_{1,n} \ar[d]^{\Phi_{n}}  \ar[r] & \mathrm{P}\Gamma_{1}^{n} 
\ar[d]^{\varphi_{n}} \ar[r] & 1 \ \ \ \ \ \ \ \ \ (\arabic{equation})\\
&& \Gamma_{1} \ar@{=}[r]\ar[d] & \Gamma_{1}\ar[d]\\
&& 1  & 1
}$$

One can repeat word by word the proof of Theorem  \ref{presentationframed}
and verify that (QR1) and (QR2) relations lift to $\widetilde{FP}_{n}(\T)$. Since $g=1$
one deduces then also (TR) relation lifts to $\widetilde{FP}_{n}(\T)$ and therefore that 
the natural map from $\widetilde{P}_{n}(\T)$ to  $\widetilde{FP}_{n}(\T)$ is actually a section.
Therefore the fact that the group $\Z^{n} $ generated by Dehn twists around boundary components is central 
implies the following result.

\begin{prop}\label{prop:thirdequiv3}
The group  $\widetilde{FP}_{n}(\T)$  is isomorphic to $\Z^n \oplus \widetilde{P}_{n}(\T)$. 
\end{prop}

\begin{prop}\label{prop:thirdequiv4}
The group  $\widetilde{FP}_{n}(\T)$  is isomorphic to $\displaystyle \frac{FP_{n}(\T)}{\Z\oplus\Z}$.
 \end{prop}
\begin{proof}
It follows from \repr{thirdequiv3} and \reth{extension centrale par Zn cas pure tore}.
One can easily deduce the result also considering the long exact sequence for $\Sigma_{g,b}=\T$
in the proof of \repr{framed braids as mapping classes}
and remarking  that $\pi_{1}\bigl(\diff{\T}\bigr)=\pi_1(\T)=\Z\oplus\Z$ (see \cite{S2}). 
\end{proof}

Finally, let $\chi_n: \Gamma_{1,n} \to \mathrm{P}\Gamma_1^{2n}$ be the injective map defined in Section \ref{section:frameddehn}.

\begin{prop}\label{prop:thirdequiv2}
The group  $\chi_n(\widetilde{FP}_n(\T))$ coincides with ${\displaystyle 
\build{\cap}{j=1}{n}C_{\widetilde{P}_{2n} (\T)}(\tau_{\gamma_j})}$,
where $C_{\widetilde{P}_{2n} (\T)}(\tau_{\gamma_j})$ is the centralizer of the Dehn twist  $\tau_{\gamma_j}$ in $\widetilde{P}_{2n} (\T)$.
\end{prop}
\begin{proof}
The proof is the same as the proof of \repr{thirdequiv} using Remark 7.
\end{proof}


\vspace{10pt}

\noindent PAOLO BELLINGERI, Universit\'{e} de Caen, CNRS UMR 6139, LMNO, Caen, 14000 (France).
Email: paolo.bellingeri@math.unicaen.fr

\vspace{5pt}
\noindent SYLVAIN GERVAIS, Universit\'{e} de Nantes, CNRS-UMR 6629, 
Laboratoire Jean Leray, 2, rue de la Houssini\`{e}re, F-44322 NANTES 
cedex 03 (France)

\end{document}